\documentclass[sn-mathphys,Numbered]{sn-jnl}

\usepackage{graphicx}%
\usepackage{multirow}%
\usepackage{amsmath,amssymb,amsfonts}%
\usepackage{amsthm}%
\usepackage{mathrsfs}%
\usepackage[title]{appendix}%
\usepackage{xcolor}%
\usepackage{textcomp}%
\usepackage{manyfoot}%
\usepackage{booktabs}%
\usepackage{algorithm}%
\usepackage{algorithmicx}%
\usepackage{algpseudocode}%
\usepackage{listings}%

\newtheorem{theorem}{Theorem}
\newtheorem{proposition}[theorem]{Proposition}

\newtheorem{remark}{Remark}

\newtheorem{lemma}{Lemma}
\newtheorem{corollary}{Corollary}

\DeclareMathOperator{\supp}{supp}

\begin{document}

\title[On spectral measures and convergence rates in von Neumann's Ergodic Theorem]{On spectral measures and convergence rates in von Neumann's Ergodic Theorem}

\author*[1]{ \sur{Moacir Aloisio}}

\author[2]{\sur{Silas L. de Carvalho}}

\author[3]{\sur{César R. de Oliveira}}

\author[4]{\sur{Edson Souza}}

\affil*[1]{\orgdiv{Corresponding author. Email: ec.moacir@gmail.com, Departamento de Matemática e Estatística}, \orgname{UFVJM}, \orgaddress{ \city{Diamantina}, \postcode{39100-000}, \state{MG}, \country{Brazil}}}

\affil[2]{\orgdiv{Email: silas@mat.ufmg.br, Departamento de Matemática}, \orgname{UFMG}, \orgaddress{ \city{Belo Horizonte}, \postcode{30161-970}, \state{MG}, \country{Brazil}}}

\affil[3]{\orgdiv{Email: oliveira@dm.ufscar.br, Departamento de Matemática}, \orgname{UFSCar}, \orgaddress{ \city{São Carlos}, \postcode{13560-970}, \state{SP}, \country{Brazil}}}

\affil[4]{\orgdiv{Email: edsonilustre@yahoo.com.br, Departamento de Matemática}, \orgname{UEA}, \orgaddress{\city{Manaus}, \postcode{369067-005}, \state{AM}, \country{Brazil}}}


\abstract{We show that the power-law decay exponents in von Neumann's Ergodic Theorem (for discrete systems) are the pointwise scaling exponents of a spectral measure at the spectral value~$1$. In this work we also prove that, under an assumption of weak convergence, in the absence of a spectral gap, the convergence rates of the time-average in von Neumann's Ergodic Theorem depend on sequences of time going to infinity.}


\maketitle

\section{Introduction}\label{sec1}

\subsection{Contextualization and main results} 

\noindent Let $U$ be a unitary operator on a separable complex Hilbert space $\mathcal{H}$ and let $P^U$ be its resolution of the identity (recall details from definition of $P^U$ in Subsection \ref{secPU}). For each $\psi \in \mathcal{H}$, set 
\begin{equation*}
\psi^* := P^U(\{1\}) \psi.
\end{equation*} 
Note that $P^U(\{1\})$ is the orthogonal projection onto the closed subspace $I(U):= \{\varphi \mid U\varphi= \varphi\}$. 

Now, recall the following classical result.  
  
\begin{theorem}[von Neumann's Ergodic Theorem]  Let $U$ be a unitary operator on  $\mathcal{H}$. Then, for each $\psi \in \mathcal{H}$, 
\begin{eqnarray}\label{eq002}
\lim_{K \to \infty} {\frac{1}{K}} \displaystyle\sum_{j=0}^{K-1} U^j\psi = \psi^*.
\end{eqnarray}
\end{theorem}

The study of convergence rates for ergodic theorems is a long-established subject in spectral theory and dynamical systems (see \cite{Ben,CarvalhoEandDS,Kachurovskii01,Kachurovskii1,Kachurovskii02,Kachurovskii2,Kachurovskii03,Podvigin,von} and references therein).  In the particular case of the von Neumann's Ergodic Theorem (vNET), some authors have obtained (power-law) convergence rates for the time-average in relation~\eqref{eq002} in  the absence of a spectral gap at $z=1$ (that is, in case $z=1$ is not an isolated point of the spectrum of $U$); most of these works are motivated by possible applications to Koopman operators  (see \cite{Kachurovskii01,Kachurovskii1,Kachurovskii02,Kachurovskii2,Kachurovskii03} and references therein; see also \cite{Ben} for a discussion involving continuous dynamical systems). 

On the other hand, if there is a spectral gap at $z=1$, it is well  known that the convergence rates for the time-average in relation~\eqref{eq002} are uniform in $\psi$ and at least of order $K^{-1}$; in what follows, $\sigma(U)$ stands for the spectrum of~$U$, a subset of $\partial \mathbb{D} := \{z \in \mathbb{C} \mid |z| = 1\}$.

\begin{theorem}[von Neumann's Ergodic Theorem with a spectral gap at $z=1$]\label{spectralgap} Let $U$ be a unitary operator on  $\mathcal{H}$. If there exists $0<\gamma < \pi$ such that $\sigma(U) \subset \{1\} \cup \{e^{i\theta} \mid \theta \in (-\pi,-\gamma] \cup (\gamma,\pi] \}$, then for every $K \in\mathbb{N}$,
\begin{eqnarray*}
\Bigg\|{\frac{1}{K}} \displaystyle\sum_{j=0}^{K-1} U^j - P^U(\{1\}) \Bigg\|   \leq \frac{4}{\gamma K}.
\end{eqnarray*} 
\end{theorem}

Although the proof of Theorem \ref{spectralgap} is a consequence of Lemma \ref{mainlemma} below, for the convenience of the reader, we present it in Appendix~\ref{appendix}. In what follows, $\mu_{\psi - \psi^*}^U$ stands for the spectral measure of $U$ with respect to $\psi - \psi^*$ on $\partial \mathbb{D}$. 

\begin{lemma}[von Neumann's Lemma]\label{mainlemma} Let  $U$ be a unitary operator on  $\mathcal{H}$. Then, for each $\psi \in \mathcal{H}$ and each $K \in\mathbb{N}$,
\begin{eqnarray}\label{eq03}
\Bigg\|{\frac{1}{K}} \displaystyle\sum_{j=0}^{K-1} U^j\psi -  \psi^* \Bigg\|^2 &=&\frac{1}{K^2}   \int_{ \partial \mathbb{D} } \left| \frac{z^K -1}{z-1} \right|^2 d\mu_{\psi - \psi^*}^U(z).
\end{eqnarray}
\end{lemma}

Lemma~\ref{mainlemma} is a  consequence of the Spectral Theorem; its proof is also presented in Appendix~\ref{appendix}. 

\begin{remark}\label{russos}{\rm  Relation \eqref{eq03} indicates that the convergence rates of the time-average in~\eqref{eq002} depend only on  local scale properties of $\mu_{\psi - \psi^*}^U$  at $z=1$. In this context, this relation arises from an argument due to von Neumann, presented in \cite{von} to prove the Quasi-Ergodic Hypothesis, and it was used more recently by Kachurovskii, Podvigin and Sedalishchev~\cite{Kachurovskii01,Kachurovskii1,Kachurovskii2} to relate the local scale properties of $\mu_{\psi - \psi^*}^U$  at $z=1$ to the power-law convergence rates in~\eqref{eq002} (see Theorem \ref{russiantheorem} ahead).}
\end{remark}

Next, we recall the result due to Kachurovskii \cite{Kachurovskii01}, mentioned in Remark~\ref{russos}, that relates the (power-law) convergence rates in vNET  to the local scale properties of $\mu_{\psi - \psi^*}^U$ at $z=1$, in case there is no spectral gap.

In what follows, for each $0<\epsilon \leq \pi$, set
\[A_\epsilon:= \big\{e^{ i \theta} \mid -\epsilon < \theta \leq \epsilon \big\}.\]

\begin{theorem}[Theorem 3. in \cite{Kachurovskii01}]\label{russiantheorem} Let $U$ be a unitary operator on $\mathcal{H}$, $\psi \in \mathcal{H}$ and $0<\alpha<2$.
\begin{enumerate} 
\item[{\rm i)}]  If there exists $C_\psi>0$ so that, for every $0<\epsilon \leq \pi$, $\mu_{\psi- \psi^*}^U(A_\epsilon)  \leq C_\psi \epsilon^{\alpha}$, then there exists ${\widetilde{C}_\psi}>0$ such that, for every $K \in\mathbb{N}$,
\begin{eqnarray*}
\bigg\| {\frac{1}{K}} \displaystyle\sum_{j=0}^{K-1} U^j\psi -  \psi^* \bigg\|^2 \leq  \frac{{\widetilde{C}_\psi}}{K^{\alpha}}\,.
\end{eqnarray*}
 \item[{\rm ii)}] Conversely, if there exists $C_\psi >0$ such that for every $K \in\mathbb{N}$,
\begin{eqnarray*}
\bigg\|{\frac{1}{K}} \displaystyle\sum_{j=0}^{K-1} U^j\psi - \psi^* \bigg\|^2  \leq \frac{{C_\psi}}{K^{\alpha}},
\end{eqnarray*}
then there exists ${\widetilde{C}_\psi}>0$ so that for every $0<\epsilon \leq \pi$, $\mu_{\psi- \psi^*}^U(A_\epsilon)  \leq {\widetilde{C}_\psi} \epsilon^{\alpha}$.
\end{enumerate} 
\end{theorem}

\begin{remark}{ \rm Although Theorem~\ref{russiantheorem} was originally stated for Koopman operators,  its proof holds for any unitary operator. Namely, the proof is essentially a consequence of Lemmas 1.~and 2.~in \cite{Kachurovskii2} (see also Lemma \ref{mainlemma1} ahead), which also hold for abstract operators. It is worth mentioning that a refinement of Theorem 3.~in \cite{Kachurovskii01}, with optimal constants for which the result is valid, is presented by Theorem 1 in \cite{Kachurovskii2} .}
\end{remark}

It is clear from Lemma \ref{mainlemma} and Theorem \ref{russiantheorem} that, in this setting, the behavior of the power-law convergence rates in vNET  depends only on the local scale properties of $\mu_{\psi - \psi^*}^U$ at $z=1$. Taking this into account, our first result in this work shows (under some mild assumptions) that these convergence rates  are indeed ruled by the lower and upper \textit{pointwise exponents}  of $\mu_{\psi - \psi^*}^U$ at $z=1$. 

Recall that the lower and upper pointwise exponents of a finite (positive) Borel measure on $\partial \mathbb{D}$,  say $\mu$,  at $z=1$ are defined, respectively, by  
\[d_\mu^-(1) := \liminf_{\epsilon \downarrow 0} \frac{\ln \mu (A_\epsilon)}{\ln \epsilon} \quad{\rm  and }\quad d_\mu^+(1) := \limsup_{\epsilon \downarrow 0} \frac{\ln \mu (A_\epsilon)}{\ln \epsilon}\,,\]
if, for all small enough $\epsilon>0$,  $\mu(A_\epsilon)> 0$; one sets $d_\mu^{\mp}(w) := \infty$, otherwise.

\begin{theorem} \label{maintheorem} Let $U$ be a unitary operator on $\mathcal{H}$, $0\neq\psi \in \mathcal{H}$ and suppose that $d^+_{\mu_{\psi- \psi^*}^U}(1) \leq 2$. Then, 
\begin{equation}\label{eq0020}
\liminf_{K \to \infty}\frac{\ln \bigg\|{\dfrac{1}{K}} \displaystyle\sum_{j=0}^{K-1} U^j\psi -\psi^* \bigg\|^2 }{-\ln K} =  d^-_{\mu_{\psi- \psi^*}^U}(1),
\end{equation}
\end{theorem}
\begin{equation}\label{eq0010}
\limsup_{K \to \infty}\frac{\ln \bigg\|{\dfrac{1}{K}} \displaystyle\sum_{j=0}^{K-1} U^j\psi -\psi^* \bigg\|^2 }{-\ln K} =  d^+_{\mu_{\psi- \psi^*}^U}(1).
\end{equation}

\begin{remark} {\rm  Note that if $d^-_{\mu_{\psi- \psi^*}^U}(1)=d^+_{\mu_{\psi- \psi^*}^U}(1) =d \leq 2$, then it follows from Theorem \ref{maintheorem} that for each $\epsilon>0$ and each sufficiently large $K$, one has
\begin{eqnarray*}
K^{-\epsilon-d} \leq \bigg\| {\frac{1}{K}} \displaystyle\sum_{j=0}^{K-1} U^j\psi -  \psi^* \bigg\|^2 \leq K^{\epsilon-d}.
\end{eqnarray*}
We note that if $d=2$, then (in general) one cannot guarantee that for all sufficiently large $K$, 
\begin{eqnarray*}
\bigg\| {\frac{1}{K}} \displaystyle\sum_{j=0}^{K-1} U^j\psi -  \psi^* \bigg\|^2 \leq K^{-2};
\end{eqnarray*}
see, e.g., Remarks 2 and 3 in \cite{Kachurovskii01}. Examples where $d^-_{\mu_{\psi- \psi^*}^U}(1)=d^+_{\mu_{\psi- \psi^*}^U}(1) =2$ can be found on page~661 in \cite{Kachurovskii01} (see also \cite{Sinai}).} 
\end{remark}

\begin{remark}\label{remkoop}{ \rm It is well known that for Koopman operators, for every $\psi\in\mathcal{H}$ such that $\psi-\psi^* \not = 0$, the respective time-average presented in vNET does not converge to $\psi^*$ faster than $K^{-2}$
    (see, for instance, Corollary 5 in \cite{Gaposhkin}). In this sense, since most of the possible applications of Theorem \ref{maintheorem} are related to Koopman operators, the  hypothesis that $d^+_{\mu_{\psi- \psi^*}^U}(1) \leq 2$ is quite reasonable.}
\end{remark}

\begin{remark}{ \rm It is worth underlying that the problem of obtaining the actual values of the decaying exponents of time-averages in terms of the dimensional properties of measures in dynamical systems  and quantum dynamics is  very natural and recurring in the literature; see, for instance, \cite{Aloisio,Barbaroux1,Barros,Barros2,CarvalhoEandDS,Holschneider}.  In this context, Theorem~\ref{maintheorem} does add something to the body of knowledge on power-law convergence rates in vNET; namely, it establishes an explicit relation between the power-law convergence rates of such time-average and the local scale spectral properties of the unitary operator $U$ at $z=1$.}
\end{remark}

By Theorem \ref{maintheorem}, if  for some $\psi\in\mathcal{H}$ one has
\begin{equation*}
  0 \leq d^-_{\mu_{\psi- \psi^*}^U}(1) < d^+_{\mu_{\psi- \psi^*}^U}(1) \leq 2,
  \end{equation*}
then   the power-law convergence rates  (related to~$\psi$)  in vNET actually depend on sequences of time going to infinity.

Write $L:= UP^U(\partial\mathbb{D}\setminus\{1\})$;  in this work, we also prove that the condition~ $$d^-_{\mu_{\psi- \psi^*}^U}(1) < d^+_{\mu_{\psi- \psi^*}^U}(1)$$ is satisfied for a unitary operator~$U$ and for a $G_\delta$ dense set of vectors $\psi\in\mathcal{H}$, if $L^j$ converges weakly to zero  (i.e., for each pair $\psi, \varphi \in \mathcal{H}$, $\langle L^j \psi,\varphi \rangle \to 0$ as $j \to \infty$) and if~$1$ is an accumulation point of $\sigma(U)$.

\begin{theorem}\label{maintheorem2}  Let $U$ be a unitary operator on $\mathcal{H}$ such that~$1$ is an accumulation point of its spectrum~$\sigma(U)$. If~$L^j$  converges to zero in the weak operator topology, then the set $\mathcal{G}(U)$ of $\psi \in \mathcal{H}$ such that, for all $\epsilon >0$,
\begin{eqnarray}\label{eqeqeq0101a}
\limsup_{K \to \infty} K^\epsilon\, \bigg\|{\frac{1}{K}} \displaystyle\sum_{j=0}^{K-1} U^j \psi - \psi^* \bigg\|^2 = \infty
\end{eqnarray}
and
\begin{eqnarray}\label{eqeqeq0101b}
\liminf_{K \to \infty} K^{2-\epsilon}\, \bigg\|{\frac{1}{K}} \displaystyle\sum_{j=0}^{K-1} U^j \psi - \psi^* \bigg\|^2 = 0,
\end{eqnarray}
is a dense $G_\delta$ set in $\mathcal{H}$.
\end{theorem}

\begin{corollary}\label{corollaymain} Under the same hypotheses of  Theorem \ref{maintheorem2}, if $\psi \in \mathcal{G}(U)$, then 
\[0=d^-_{\mu_{\psi- \psi^*}^U}(1) < 2 \leq d^+_{\mu_{\psi- \psi^*}^U}(1).\] 
\end{corollary}

\begin{remark}
{\rm The results presented in this work are analogous to those  obtained by the first three authors in \cite{Aloisio} on the convergence rates of normal semigroups (see also \cite{Alo,Carvalho1,CarvalhoEandDS}). The main technical difference is that, for normal semigroups, it is enough to analyze the (self-adjoint) real part of the generator in order to evaluate the power-law convergence rates, whereas here one needs to analyze  spectral properties of unitary operators.

We note that an adaptation of the arguments in~\cite{Aloisio} could be employed to study the convergence rates in vNET for continuous dynamical systems, since this case leads naturally to self-adjoint generators.}
\end{remark}

\subsection{Notations and organization of the text}\label{secPU}

Some words about notation: $U$ will always denote a unitary operator acting on the separable complex Hilbert space~$\mathcal H$; we denote its spectrum by $\sigma(U)$. 

Recall that, by the Spectral Theorem \cite{Rudin} (see also \cite{Damanikunit}), there is a unique map that associates with each bounded Borel measurable function $f$ on the unit circle $\partial \mathbb{D}$ a bounded normal operator, that is,
\begin{equation*}
    P^U(f) = \int_{ \partial \mathbb{D}} f(z) \, dP^U(z),
\end{equation*}
that satisfies, for each $\varphi \in \mathcal{H}$,
\begin{equation*}
    \Vert P^U(f)\varphi\Vert^2 =  \bigg\| \int_{ \partial \mathbb{D}} f(z) \, dP^U(z)\varphi \bigg\|^2 = \int_{ \partial \mathbb{D}} |f(z)|^2 d\mu_{\varphi}^U(z);
\end{equation*}
here, $\mu_{\varphi}^U$ is a positive finite Borel measure supported on  $\partial \mathbb{D}$, the so-called spectral measure of $U$ with respect to $\varphi$.

One can also define, for each Borel set $\Lambda \subset \partial \mathbb{D}$, the spectral projection $$P^U(\Lambda) := P^U(\chi_\Lambda);$$ the family $\{P^U(\Lambda)\}_{\Lambda}$  is called the resolution of the identity of $U$ (see~\cite{Oliveira,Rudin}).

The work is organized as follows. In Section \ref{secKoopmanoperators}, we illustrate our general results by presenting an explicit application of Theorem~\ref{maintheorem2} to Koopman operators. The proof of Theorems~\ref{maintheorem} and~\ref{maintheorem2}  and of Corollary \ref{corollaymain} are left to Section \ref{secproofs}. In Appendix~\ref{appendix}, we prove Lemma \ref{mainlemma} and Theorem \ref{spectralgap}.  


\section{Application  to Koopman operators}\label{secKoopmanoperators}

\noindent Let $(X, \mathcal{A},m)$ be a Lebesgue space with continuous measure $m$ (i.e., a non-atomic probability space obtained by the completion of a Borel measure on a complete separable metric space); for every automorphism $T$ on $X$, one defines the Koopman operator related to $T$ by the law 
\[ U_T : L_m^2(X) \to  L_m^2(X),   \quad (U_T f )(x) := f (T x),\quad x \in X.\]

We say that $\sigma(U_T)$ is purely absolutely continuous  on $\partial \mathbb{D}\setminus\{1\}$ if $U_T$ is purely absolutely continuous on 
\[
\big\{f\in  L_m^2(X)\mid \textstyle\int_X f(x)\,dm(x)=0\big\};
\] observe that $1$ is always an eigenvalue of the Koopman operator, since
$U_T \varphi = \varphi$ for every constant function $\varphi$.

\begin{theorem}\label{application1} Let $T$ be an automorphism on~$X$ with purely absolutely continuous spectrum on $\partial \mathbb{D}\setminus\{1\}$ such that $1$ is an accumulation point of $\sigma(U_T)$. Then, the set of vectors $\psi \in L_m^2(X)$ so that, for all $\epsilon >0$,
\begin{eqnarray*}
\limsup_{K \to \infty} K^\epsilon \bigg\|{\frac{1}{K}} \displaystyle\sum_{j=0}^{K-1} U_T^j \psi - \psi^* \bigg\|^2 = \infty
\end{eqnarray*}
and
\begin{eqnarray*}
\liminf_{K \to \infty} K^{2-\epsilon} \bigg\|{\frac{1}{K}} \displaystyle\sum_{j=0}^{K-1} U_T^j \psi - \psi^* \bigg\|^2 = 0,
\end{eqnarray*}
is a dense $G_\delta$ set in $L_m^2(X)$.
\end{theorem}

\begin{remark}\end{remark}
{\rm    
\begin{enumerate}
    \item [i)] The result stated in Theorem  \ref{application1} applies to $K$-automorphisms, in particular to Bernoulli shifts \cite{Rohlin}.

    \item[ii)] Note that under the assumptions of Theorem \ref{application1}, it follows from 
      Corollary \ref{corollaymain} that for a typical $\psi\in L_m^2(X)$, $d^-_{\mu_{\psi- \psi^*}^T}(1)=0$  and $d^+_{\mu_{\psi- \psi^*}^T}(1)\geq2$.
\end{enumerate}}

The proof of Theorem \ref{application1} is a direct consequence of Theorem \ref{maintheorem2} and of the next result. 

\begin{proposition}\label{tecprop} Let $U$ be a unitary operator on $\mathcal{H}$ with absolutely continuous spectrum in $\partial \mathbb{D}\setminus \{1\}$ and  $L=UP^U(\partial \mathbb{D}\setminus \{1\})$.  Then, $L^j$ converges to zero in the weak operator topology as $j \to \infty$.
\end{proposition}

\begin{proof}  For each $\phi \in \mathcal{H}$, let 
\[\mathcal{H}_\phi:= \big\{P^U(g)\phi \, \mid \, g \in L^2_{\mu_\phi^U}(\partial \mathbb{D}) \big\}\]
be the cyclic subspace generated by $\phi$, and so $P_\phi$ be the corresponding orthogonal projection onto $\mathcal{H}_\phi$ and $F: \mathcal{H}_\phi \to L^2_{\mu_\phi^U}(\partial \mathbb{D})$ be the unitary operator given by the law
 $F(P^U(g)\phi) = g$ (note that $L=P^U(g \chi_{\partial \mathbb{D}\setminus \{1\}})$, with $g(z) = z$).

Fix $\psi,\varphi\in\mathcal{H}$ and let $h \in L^2_{\mu_\psi^U}(\partial \mathbb{D})$ be such that $P_\psi \varphi= P^U(h)\psi$. Let also $\{g_n\}$  be a  sequence in the space of bounded Borel functions in $\partial \mathbb{D}$, which is a dense subspace of  $L^2_{\mu_\psi^U}(\partial \mathbb{D})$, with $||g_n - h||_{L^2_{\mu_\psi^U}(\partial \mathbb{D})} \to 0$ as $n \to \infty$. 

Now, since $\mu_\psi^U(\cdot\cap\partial \mathbb{D}\setminus \{1\})$ is absolutely continuous, it follows from Radon-Nikodym Theorem that there exists $f \in L^1(\partial \mathbb{D}\setminus \{1\})$ such that $d\mu_\psi^U(z) = f(z) dz$. Therefore, $g_n f \in L^1(\partial \mathbb{D}\setminus \{1\})$ for each $n\in\mathbb{N}$, and so, it follows from Riemann-Lebesgue Lemma that for every $n \in \mathbb{N}$,   
\begin{eqnarray}\label{vaiazero}
\nonumber   \lim_{j \to \infty} \big\langle  L^j \psi, P^U(g_n)\psi \big\rangle  &=&    \lim_{j \to \infty}  \int_{ \partial \mathbb{D}\setminus \{1\}} z^j  g_n(z) d\mu_\psi^U(z)\\   \nonumber &=&  \lim_{j \to \infty}  \int_{ \partial \mathbb{D}\setminus \{1\}} z^j  g_n(z) f(z) dz \\ &=&  \lim_{j \to \infty}  \int_0^1 e^{2i\pi \theta j} g_n (e^{2i\pi \theta}) f(e^{2i\pi \theta})  d\theta =0;
\end{eqnarray}
 we have used the fact that the Lebesgue measure on $\partial\mathbb{D}$ is the pushforward, with respect to the map $h:[0,1)\rightarrow \partial\mathbb{D}$, $h(\theta)=e^{2\pi i\theta}$, of the Lebesgue measure on $[0,1)$. 

On the other hand, for each $j \geq 0$, one has
\begin{eqnarray*}
\nonumber \big|\big\langle L^j \psi,\varphi \big\rangle - \big\langle L^j \psi, P^U(g_n)\psi\big\rangle\big|  &=&   \big|\big\langle P_\psi  L^j \psi,\varphi \big\rangle - \big\langle L^j \psi, P^U(g_n)\psi\big\rangle \big|\\ &=& \big| \big\langle L^j \psi, P_\psi\varphi \big\rangle - \big\langle L^j \psi, P^U(g_n)\psi\big\rangle \big| \\ &=& \big| \big\langle L^j \psi, P^U(h)\psi - P^U(g_n)\psi\big\rangle \big| \\  &\leq& \Vert\psi\Vert\, \Vert P^U(h)\psi -  P^U(g_n)\psi \Vert  \\ &=& \Vert\psi\Vert \, \Vert h -  g_n\Vert_{L^2_{\mu_\psi^U}(\partial \mathbb{D})},
\end{eqnarray*}
and so
\begin{eqnarray*}
  \lim_{n\to\infty} \big|\big\langle L^j \psi,\varphi \big\rangle - \big\langle L^j \psi, P^U(g_n)\psi\big\rangle \big|=0;
  \end{eqnarray*}
note that this convergence is uniform on $j$. Therefore, by combining this result with relation~\eqref{vaiazero} and Moore-Osgood Theorem, one gets
\begin{eqnarray*}
\nonumber \lim_{j \to \infty} \big\langle L^j \psi,\varphi \big\rangle &=&   \lim_{j \to \infty} \lim_{n \to \infty}  \big\langle  L^j \psi, P^U(g_n)\psi \big\rangle \\  &=&   \lim_{n \to \infty}  \lim_{j \to \infty} \big\langle  L^j \psi, P^U(g_n)\psi \big\rangle =0.
\end{eqnarray*}
\end{proof}


\section{Proofs}\label{secproofs}

\subsection{Proof of Theorem~\ref{maintheorem}}
\noindent We begin with some preparation.
In what follows, for each $K \in \mathbb{N}$, set
\[S_k:= \big\{e^{ i \theta} \mid -\frac{\pi}{K} < \theta \leq \frac{\pi}{K} \big\}.\]

The lemma below is the main tool in the proof of Theorem \ref{maintheorem}. Let $f :(-\pi,\pi]\rightarrow \partial\mathbb{D}$ denote the map  $f(\theta)= e^{ i\theta}$ and let $f^*\mu$ be the pull-back, with respect to~$f$, of the positive Borel measure $\mu$ on $\partial \mathbb{D}$.

\begin{lemma}[Lemma 2. in \cite{Kachurovskii2}]\label{mainlemma1} For every $K \in \mathbb{N}$, the following inequality holds:
\begin{eqnarray}\label{le0b}\nonumber
\frac{1}{{K}^2}   \int_{-\pi}^\pi \left| \frac{\sin(\theta {K}/2)}{\sin(\theta/2)} \right|^2 d(f^*\mu_{\psi - \psi^*}^U)(\theta) &\leq& \frac{\mu_{\psi - \psi^*}^U(S_1)}{{K}^2}\\  &+& \frac{1}{K^2}\sum_{j=1}^{K-1}  (2j+1) \mu_{\psi - \psi^*}^U(S_j)
\end{eqnarray}
\end{lemma}

\ 

 Let us proceed to the proof of Theorem~\ref{maintheorem}. It follows from Lemma~\ref{mainlemma} that, for each $K \in\mathbb{N}$, 
\begin{eqnarray*}
  \nonumber \bigg\|{\frac{1}{K}} \displaystyle\sum_{j=0}^{K-1} U^j\psi -  \psi^* \bigg\|^2 &=& \frac{1}{K^2}   \int_{ \partial \mathbb{D} } \left| \frac{z^K -1}{z-1} \right|^2 d\mu_{\psi - \psi^*}^U(z)\\ \nonumber &=&\frac{1}{{K}^2}   \int_{-\pi}^\pi \left| \frac{\sin(\theta {K}/2)}{\sin(\theta/2)} \right|^2 d(f^*\mu_{\psi - \psi^*}^U)(\theta)\\
  &=& \frac{1}{K^2}   \int_{-\epsilon}^{\epsilon}\left| \frac{\sin(\theta {K}/2)}{\sin(\theta/2)} \right|^2 d(f^*\mu_{\psi - \psi^*}^U)(\theta)\\ &+&  \frac{1}{K^2}   \int_{(-\pi,-\epsilon] \cup (\epsilon,\pi]} \left| \frac{\sin(\theta {K}/2)}{\sin(\theta/2)} \right|^2  d(f^*\mu_{\psi - \psi^*}^U)(\theta)\\ &\geq&  \frac{1}{K^2}   \int_{-\epsilon}^{\epsilon} \left| \frac{\sin(\theta {K}/2)}{\sin(\theta/2)} \right|^2  d(f^*\mu_{\psi - \psi^*}^U)(\theta).
\end{eqnarray*}
Since for every $-\pi/2\leq \theta \leq \pi/2$,  $\displaystyle\frac{|\theta|}{2} \leq |\sin(\theta)| \leq |\theta|$, it follows that for each $K \in \mathbb{N}$ and each $0< \epsilon < 1/K$,
\begin{eqnarray*}
\bigg\|{\frac{1}{K}} \displaystyle\sum_{j=0}^{K-1} U^j\psi -  \psi^* \bigg\|^2 \geq \frac{(f^*\mu_{\psi - \psi^*}^U)(-\epsilon,\epsilon)}{4}  = \frac{\mu_{\psi - \psi^*}^U(A_\epsilon)}{4}. 
\end{eqnarray*}
Hence,
\begin{equation}\label{eq0a}
\liminf_{K \to \infty}\frac{\ln \bigg\|{\dfrac{1}{K}} \displaystyle\sum_{j=0}^{K-1} U^j\psi -\psi^* \bigg\|^2 }{-\ln K } \leq d^-_{\mu_{\psi- \psi^*}^U}(1).
\end{equation}
\begin{equation}\label{eq0b}
\limsup_{K \to \infty}\frac{\ln \bigg\|{\dfrac{1}{K}} \displaystyle\sum_{j=0}^{K-1} U^j\psi -\psi^* \bigg\|^2 }{-\ln K } \leq  d^+_{\mu_{\psi- \psi^*}^U}(1),
\end{equation}

\ 

Now, we prove the complementary inequalities in \eqref{eq0a} and \eqref{eq0b}.

\

\noindent {\bf Case $d^-_{\mu_{\psi- \psi^*}^U}(1)$:} If $d^-_{\mu_{\psi- \psi^*}^U}(1) =0$, the complementary inequality in \eqref{eq0a} follows readily. So, assume that $d^-:=d^-_{\mu_{\psi- \psi^*}^U}(1)>0$. It follows from the definition of $d^-$ that for every $0< \epsilon< d^-$, there exists $K_\epsilon\in\mathbb{N}$ such that for each $K \in \mathbb{N}$ with $K > K_\epsilon$, one has
\begin{equation}\label{lemeq01}
\mu_{\psi - \psi^*}^U(S_K) \leq K^{\epsilon-d^-}.
\end{equation}

Now, it follows from relation~\eqref{le0b} that for each $K \in \mathbb{N}$,
\begin{eqnarray}\label{eq000001}
   \nonumber \bigg\|{\frac{1}{K}} \displaystyle\sum_{j=0}^{K-1} U^j\psi -  \psi^* \bigg\|^2 &=& \frac{1}{K^2}   \int_{ \partial \mathbb{D} } \left| \frac{z^K -1}{z-1} \right|^2 d\mu_{\psi - \psi^*}^U(z)\\ \nonumber &=&\frac{1}{{K}^2}   \int_{-\pi}^\pi \left| \frac{\sin(\theta {K}/2)}{\sin(\theta/2)} \right|^2 d(f^*\mu_{\psi - \psi^*}^U)(\theta)\\ &\leq& \frac{\mu_{\psi - \psi^*}^U(S_1)}{K^2} + \frac{4}{K^2}\sum_{j=1}^{K-1} j \mu_{\psi - \psi^*}^U(S_j).
\end{eqnarray}
Hence, by combining relations \eqref{lemeq01}  and \eqref{eq000001}, one concludes that there exists $C=C(\epsilon)>0$ such that for each $K \in \mathbb{N}$ with $K > K_\epsilon$, 
\begin{eqnarray}\label{eqqq0000a1}
  \nonumber \bigg\|{\frac{1}{K}} \displaystyle\sum_{j=0}^{K-1} U^j\psi -  \psi^* \bigg\|^2  &\leq&  \frac{C}{K^2}  + \frac{4}{K^2}\sum_{j=K_\epsilon}^{K-1}   j^{1-(d^--\epsilon)}  \\ &\leq&   \frac{C}{K^2} + \frac{4}{K^2}\sum_{j=1}^{K-1}  j^{1-(d^--\epsilon)}  .
\end{eqnarray}

Let $L\in \mathbb{N}$. If $d^--\epsilon \in (0,1)$, then (see page 1111 in \cite{Kachurovskii2})
\begin{eqnarray*}
 \frac{2}{L^2}\sum_{j=1}^{L-1}   j^{1-(d^--\epsilon)} &\leq& \frac{2}{2- (d^--\epsilon)} L^{\epsilon -d^-} + \frac{2}{2- (d^--\epsilon)} L^{-2}.
\end{eqnarray*}
If $ d^--\epsilon =1$, then 
\begin{eqnarray*}
 \frac{2}{L^2}\sum_{j=1}^{L-1}   j^{1-(d^--\epsilon)}  = 2L^{-1} -2L^{-2}.
\end{eqnarray*}
If $d^--\epsilon \in (1,2)$, then (see page 1112 in \cite{Kachurovskii2})
\begin{eqnarray*}
 \frac{2}{L^2}\sum_{j=1}^{L-1}   j^{1-(d^--\epsilon)} &\leq& \frac{2}{2- (d^--\epsilon)} L^{\epsilon -d^-}\\ &+& \biggr(2-\frac{2}{2- (d^--\epsilon)}\biggr) L^{-2} + 2L^{-2-(d^--\epsilon)}.
\end{eqnarray*}
Therefore, by combining the discussion above with relation~\eqref{eqqq0000a1}, one concludes that there exists $D=D(\epsilon,d^-) >0$ such that  for each $K \in \mathbb{N}$ with $K > \Tilde{K} = \Tilde{K}(\epsilon,d^-) $, 
\begin{eqnarray}\label{eqmain}
\bigg\|{\frac{1}{K}} \displaystyle\sum_{j=0}^{K-1} U^j\psi -  \psi^* \bigg\|^2  &\leq&  \frac{C}{K^2} + D K^{\epsilon -d^-}.
\end{eqnarray}

Finally,  since  $0<d^- - \epsilon < 2$ by hypothesis, it follows from \eqref{eqmain} that   for each sufficiently large~$K$, 
\begin{eqnarray*}
  \nonumber \bigg\|{\frac{1}{K}} \displaystyle\sum_{j=0}^{K-1} U^j\psi -  \psi^* \bigg\|^2 &\leq& \max\{C,D\} K^{\epsilon -d^-}.
\end{eqnarray*}
Hence,
\begin{equation*}
\liminf_{K \to \infty}\frac{\ln \bigg\|{\dfrac{1}{K}} \displaystyle\sum_{j=0}^{K-1} U^j\psi -\psi^* \bigg\|^2 }{-\ln K } \geq d^- - \epsilon = d^-_{\mu_{\psi- \psi^*}^U}(1) - \epsilon .
\end{equation*}
Since  $0<\epsilon<d^-_{\mu_{\psi- \psi^*}^U}(1)$ is arbitrary, the complementary inequality in \eqref{eq0a} follows. 

\

\noindent {\bf Case $d^+_{\mu_{\psi- \psi^*}^U}(1)$}: Let $\epsilon >0$.  It follows from the definition of $d^+:=d^+_{\mu_{\psi- \psi^*}^U}(1)$ that there exists a monotonically increasing sequence $\{K_l\}$ of natural numbers such that for each $l\in\mathbb{N}$, 
\begin{equation}\label{eq04cccc}
\mu_{\psi - \psi^*}^U(S_{K_l}) \leq K_l^{\frac{\epsilon}{2} - d^+}.
\end{equation}
On the other hand, it also follows from the definition of $d^+_{\mu_{\psi- \psi^*}^U}(1) \leq 2$ that there exists $K_\epsilon\in\mathbb{N}$ such that for each $K \in \mathbb{N}$ with $K > K_\epsilon$,
\begin{equation}\label{eq04dd}
\frac{1}{K^{2+\frac{\epsilon}{4}}} \leq \mu_{\psi - \psi^*}^U(S_K).
\end{equation}
In particular, 
\[ \liminf_{K \to \infty} {K}^{2+\frac{\epsilon}{2} } \mu_{\psi - \psi^*}^U(S_{K}) \geq  \liminf_{K \to \infty} {K}^{\frac{\epsilon}{4}} = \infty.\] 
So, we can extract from $\{K_l\}$ a monotonically increasing subsequence $\{K_{l_m}\}$ such that
\begin{equation}\label{eq05dd}
  \max_{1 \leq j \leq K_{l_m}} \{j^{2+ \frac{\epsilon}{2} } \mu_{\psi - \psi^*}^U(S_{j})\} =  K_{l_m}^{2+\frac{\epsilon}{2}} \mu_{\psi - \psi^*}^U(S_{K_{l_m}}).
\end{equation}
Thus, by combining~\eqref{le0b},~\eqref{eq04cccc},~\eqref{eq04dd} and \eqref{eq05dd}, it follows that for each sufficiently large $m$,
\begin{eqnarray*}
  \nonumber \bigg\|{\frac{1}{K_{l_m}}} \displaystyle\sum_{j=0}^{K_{l_m}-1} U^j\psi -  \psi^* \bigg\|^2   &\leq& \frac{\mu_{\psi - \psi^*}^U(S_1)}{K_{l_m}^2} + \frac{4}{K_{l_m}^2}\sum_{j=1}^{K_{l_m}-1}  j \mu_{\psi - \psi^*}^U(S_{{j}})\\&\leq& \frac{\mu_{\psi - \psi^*}^U(S_1)}{K_{l_m}^2} + \frac{4}{K_{l_m}^2}\sum_{j=1}^{K_{l_m}}  j^{-(1+\frac{\epsilon}{2} )}j^{2+\frac{\epsilon}{2} }\mu_{\psi - \psi^*}^U(S_{{j}})\\ &\leq&  \frac{\mu_{\psi - \psi^*}^U(S_1)}{K_{l_m}^2} + 4 K_{l_m}^{\frac{\epsilon}{2} } \mu_{\psi - \psi^*}^U(S_{K_{l_m}})\sum_{j=1}^{\infty}  j^{-(1+\frac{\epsilon}{2})}\\
  &\leq&  C_\epsilon K_{l_m}^{\frac{\epsilon}{2} } \mu_{\psi - \psi^*}^U(S_{K_{l_m}})   \leq  C_\epsilon K_{l_m}^{\epsilon-d^+}, 
\end{eqnarray*}
with $C_\epsilon= \max\{\mu_{\psi - \psi^*}^U(S_1), 4\sum_{j=1}^{\infty}  j^{-(1+\frac{\epsilon}{2})}\}$. Hence, 
\begin{equation}\nonumber
\limsup_{K \to \infty}\frac{\ln \bigg\|{\dfrac{1}{K}} \displaystyle\sum_{j=0}^{K-1} U^j\psi -\psi^* \bigg\|^2 }{-\ln K } \geq d^+ -\epsilon = d^+_{\mu_{\psi- \psi^*}^U}(1) -\epsilon,
\end{equation}
and since  $\epsilon>0$ is arbitrary, the complementary inequality in \eqref{eq0b} follows.


\subsection{Proof of Theorem~\ref{maintheorem2}}
\noindent Again, we begin with some preparation.

\begin{theorem}[Theorem 1 in~\cite{Muller2}]\label{Muller}  Let $L$ be a bounded linear operator on $\mathcal{H}$ such that $L^j$ converges to zero in the weak operator topology. Suppose that $1 \in \sigma(L)$. Then, for each sequence $(a_j)_{j\ge 0}$ of positive numbers satisfying $\displaystyle\lim_{j \to \infty} a_j = 0$ and for each $\delta>0$, there exists $\eta \in \mathcal{H}$ such that $\Vert\eta\Vert < \displaystyle\sup_j\{a_j\} + \delta$  and
  \[{\rm Re}\,\langle L^j \eta,\eta \rangle > a_j,\quad\forall j \geq 0.\]
\end{theorem}

\begin{lemma}\label{teclemma2}  Let $U$ be a unitary operator on $\mathcal{H}$, $L= UP^U(\partial\mathbb{D}\setminus\{1\})$, and suppose that $L^j$  converges to zero in the weak operator topology. If $1$ is an accumulation point of $\sigma(U)$,  then for each $0<\epsilon<1$, there exists $\eta \in \mathcal{H}$ with $\Vert\eta\Vert \leq 1$ so that 
\begin{eqnarray*}
\liminf_{K \to \infty} K^\epsilon \, \bigg\|{\frac{1}{K}} \displaystyle\sum_{j=0}^{K-1} U^j \eta - \eta^* \bigg\|^2 = \infty.
\end{eqnarray*}
\end{lemma}

\begin{proof}Observe that:

\begin{enumerate}
\item[i)] $1 \in \sigma(L)$, since $1$ is an accumulation point of $\sigma(U)$.
 
\item[ii)] $L^j$ converges to zero in the weak operator topology.

\item[iii)]  By the spectral functional calculus, $UP^U(\partial \mathbb{D}\setminus \{1\}) = P^U(\partial \mathbb{D}\setminus \{1\})U$, and so, for  $j\in\mathbb{N}$, 
\begin{eqnarray*}
L^j &=& U^j(1 - P^U(\{1\}))\\&=&U^j - U^jP^U(\{1\})=U^j-P^U(\{1\}).
\end{eqnarray*}
\end{enumerate}

Let $0<\epsilon'<1$ be such that $0<2\epsilon'<\epsilon<1$. By Theorem \ref{Muller}, there exists $\eta \in \mathcal{H}$ with $\Vert\eta\Vert \leq 1$ such that, for each $j\ge0$,
\[{\rm Re}\,\langle L^j \eta,\eta \rangle > \frac{1}{(j+2)^{\epsilon'}};\]
thus, for $K \gg 1$, 
\begin{eqnarray*}{\rm Re}\,\biggr\langle {\frac{1}{K}} \displaystyle\sum_{j=0}^{K-1} L^j\eta,\eta \biggr\rangle &>& {\frac{1}{K}} \displaystyle\sum_{j=0}^{K-1}  \frac{1}{(j+2)^{\epsilon'}}\\ &\ge& \frac{1}{K} K \frac{1}{(K+2)^{\epsilon'}} \geq \frac{1}{2K^{\epsilon'}}.
\end{eqnarray*}
. By~iii) above, for  $K \gg 1$,
\begin{eqnarray*}
  {\rm Re}\,\biggr\langle {\frac{1}{K}} \sum_{j=0}^{K-1} U^j \eta - \eta^*,\eta \biggr\rangle  &=& {\rm Re}\left(\biggr\langle {\frac{1}{K}} \sum_{j=1}^{K-1} L^j \eta,\eta \biggr\rangle + \biggr\langle {\frac{1}{K}}(\eta-\eta^*),\eta \biggr\rangle\right)\\
  &>&\frac{1}{2K^{\epsilon'}} - {\rm Re}\,\biggr\langle {\frac{\eta^*}{K}},\eta \biggr\rangle.
\end{eqnarray*}
By Cauchy-Schwarz inequality,  for  $K \gg 1$,
\begin{eqnarray*} \bigg\|{\frac{1}{K}} \displaystyle\sum_{j=0}^{K-1} U^j \eta - \eta^* \bigg\|> \frac{1}{2K^{\epsilon'}} - \frac{1}{K} \geq \frac{1}{4K^{\epsilon'}}.
\end{eqnarray*}
Hence,
\begin{eqnarray*}
\liminf_{K \to \infty} K^\epsilon\, \bigg\|{\frac{1}{K}} \displaystyle\sum_{j=0}^{K-1} U^j \eta - \eta^* \bigg\|^2 = \infty,
\end{eqnarray*}
and we are done. 
\end{proof}

 Let us proceed to the proof of Theorem~\ref{maintheorem2}. Let $\epsilon>0$; since for each $K \in\mathbb{N}$, the map 
\[ {\mathcal{H}} \ni \psi\; \longmapsto K^\epsilon \bigg\|{\frac{1}{K}} \displaystyle\sum_{j=0}^{K-1} U^j \psi - \psi^* \bigg\| = K^\epsilon \bigg\|{\frac{1}{K}} \displaystyle\sum_{j=0}^{K-1} U^j \biggr\{I - P^U({\{1\}})\biggr\}\psi \bigg\|  \]
is continuous (namely, it consists of finite sums and compositions of continuous functions), it follows that 
\begin{eqnarray*}
G^+(\epsilon)&:=&\biggl\{\psi  \mid  \limsup_{K \to \infty} K^\epsilon \bigg\|{\frac{1}{K}} \displaystyle\sum_{j=0}^{K-1} U^j \psi - \psi^* \bigg\| = \infty \biggl\} \\ &=& \bigcap_{n\geq 1}  \biggl\{\psi \mid \limsup_{K \to \infty} K^\epsilon \bigg\|{\frac{1}{K}} \displaystyle\sum_{j=0}^{K-1} U^j \psi - \psi^* \bigg\| > n \biggl\} \\ &=& \bigcap_{n\geq 1} \bigcap_{l \geq 1} \bigcup_{K\geq l} \biggl\{\psi \mid  K^\epsilon \bigg\|{\frac{1}{K}} \displaystyle\sum_{j=0}^{K-1} U^j \psi - \psi^* \bigg\| > n \biggl\}
\end{eqnarray*}
is a $G_\delta$ set in ${\mathcal{H}}$. The proof that 
\begin{eqnarray*}
G^-(\epsilon):=\biggl\{\psi  \mid  \liminf_{K \to \infty} K^{2-\epsilon} \bigg\|{\frac{1}{K}} \displaystyle\sum_{j=0}^{K-1} U^j \psi - \psi^* \bigg\| = 0 \biggl\}
\end{eqnarray*}
is also a $G_\delta$ set in ${\mathcal{H}}$ is completely analogous, and so we omit it. 

\ 

\noindent {\bf Claim:} For each $\epsilon>0$, $G^+(\epsilon)$ and $G^-(\epsilon)$ are dense sets in $\mathcal{H}$.

It follows from {\bf Claim} and Baire Category Theorem that 
\[\mathcal{G}:=\bigcap_{\epsilon \in \mathbb{Q} \cap (0,\infty)} \biggr( G^+(\epsilon) \cap G^-(\epsilon) \biggr)\]
is also a dense $G_\delta$ set in ${\mathcal{H}}$. (Note that if $0<\epsilon < \delta$, so $G^\mp(\epsilon)  \subset  G^\mp(\delta)$; therefore, is enough take just the intersection at $\mathbb{Q}$). Since for each $\psi \in \mathcal{G}$ and each $\epsilon >0$,
\begin{eqnarray*}
\limsup_{K \to \infty} K^\epsilon \bigg\|{\frac{1}{K}} \displaystyle\sum_{j=0}^{K-1} U^j \psi - \psi^* \bigg\|^2 = \infty
\end{eqnarray*}
and
\begin{eqnarray*}
\liminf_{K \to \infty} K^{2-\epsilon} \bigg\|{\frac{1}{K}} \displaystyle\sum_{j=0}^{K-1} U^j \psi - \psi^* \bigg\|^2 = 0,
\end{eqnarray*}
this concludes the proof. Therefore, one just has to  prove the {\bf Claim}.

\noindent {\bf Proof of the Claim.} Fix $\epsilon>0$.

$\bullet$ $G^+(\epsilon)$ is dense in $\mathcal{H}$. Let $0\neq \psi \in \mathcal{H}\backslash G^+(\epsilon)$; so,
there exists $C_\psi>0$ such that for sufficiently large $K$,
\begin{equation}
\label{thm201}
\bigg\|\frac{1}{K} \displaystyle\sum_{j=0}^{K-1} U^j \psi - \psi^* \bigg\|\leq \frac{C_\psi}{K^{\epsilon/2}}
\end{equation}
(otherwise, $\psi \in G^+(\epsilon)$). It follows from Lemma \ref{teclemma2} that there exists $\eta\in\mathcal{H}$, with $\Vert\eta\Vert \leq 1$, 
such that
\begin{equation}\label{thm202}
\liminf_{K \to \infty} K^{\epsilon/3} \bigg\|{\frac{1}{K}} \displaystyle\sum_{j=0}^{K-1} U^j \eta - \eta^* \bigg\|^2 = \infty.
\end{equation}

Note that, by Cauchy-Schwarz inequality and triangle inequality,
\begin{eqnarray}\label{b007}
   \nonumber - {\rm Re}\, \biggr\langle {\frac{1}{K}} \displaystyle\sum_{j=0}^{K-1} U^j \psi - \psi^* , {\frac{1}{K}} \displaystyle\sum_{j=0}^{K-1} U^j \eta - \eta^* \biggr\rangle &\leq& \bigg\|{\frac{1}{K}} \displaystyle\sum_{j=0}^{K-1} U^j \psi - \psi^* \bigg\| \Vert(I - P^U({\{1\}}))\eta\Vert\\ &\leq& \bigg\|{\frac{1}{K}} \displaystyle\sum_{j=0}^{K-1} U^j \psi - \psi^* \bigg\|.
\end{eqnarray}

Finally, for each $m \in\mathbb{N}$, set $\psi_m:=\psi+ \displaystyle\frac{\eta}{m}$; naturally, $\psi_m \to \psi$ as $m \to \infty$. Moreover, by \eqref{thm201}, \eqref{thm202} and \eqref{b007}, one has for each $m\in\mathbb{N}$ and each sufficiently large $K$, 

\begin{eqnarray*} K^{\epsilon} \bigg\|{\frac{1}{K}} \displaystyle\sum_{j=0}^{K-1} U^j \psi_m - \psi_m^* \bigg\|^2  &=&  K^{\epsilon}  \bigg\|{\frac{1}{K}} \displaystyle\sum_{j=0}^{K-1} U^j \psi - \psi^* \bigg\|^2\\ &+& \frac{2K^{\epsilon} }{m} {\rm Re}\, \biggr\langle {\frac{1}{K}} \displaystyle\sum_{j=0}^{K-1} U^j \psi - \psi^* , {\frac{1}{K}} \displaystyle\sum_{j=0}^{K-1} U^j \eta - \eta^* \biggr\rangle\\  &+& \frac{K^{\epsilon}}{m^2}\bigg\|{\frac{1}{K}} \displaystyle\sum_{j=0}^{K-1} U^j \eta - \eta^* \bigg\|^2\\ &\geq& \frac{2K^{\epsilon} }{m} {\rm Re}\, \biggr\langle {\frac{1}{K}} \displaystyle\sum_{j=0}^{K-1} U^j \psi - \psi^* , {\frac{1}{K}} \displaystyle\sum_{j=0}^{K-1} U^j \eta - \eta^* \biggr\rangle\\  &+& \frac{K^{\epsilon}}{m^2}\bigg\|{\frac{1}{K}} \displaystyle\sum_{j=0}^{K-1} U^j \eta - \eta^* \bigg\|^2\\ &\geq& -\frac{2K^{\epsilon} }{m} \bigg\|{\frac{1}{K}} \displaystyle\sum_{j=0}^{K-1} U^j \psi - \psi^* \bigg\|\\  &+& \frac{K^{\epsilon}}{m^2}\bigg\|{\frac{1}{K}} \displaystyle\sum_{j=0}^{K-1} U^j \eta - \eta^* \bigg\|^2\\ &\geq& -\frac{2K^{\epsilon} }{m} \cdot \frac{C_\psi}{K^{\epsilon/2}}  + \frac{K^{\epsilon}}{m^2} \cdot \frac{1}{K^{\epsilon/3}}.
\end{eqnarray*}
Thus, for every $m \in\mathbb{N}$, 
\[\limsup_{K \to \infty} K^{\epsilon} \bigg\|{\frac{1}{K}} \displaystyle\sum_{j=0}^{K-1} U^j \psi_m - \psi_m^* \bigg\|^2 = \infty,\]
from which follows that $G^+(\epsilon)$ is dense in $\mathcal{H}$. \hfill $\qedsymbol$

\ $\bullet$ $G^-(\epsilon)$ is dense in $\mathcal{H}$. Let $\psi \in \mathcal{H}$ and set, for each $n \in \mathbb{N}$,  
\[
A_n := \{1\} \,\cup\,  \biggr\{e^{i \theta} \, : \, \theta \in \biggr(-\pi,-\frac{1}{n}\biggr] \cup \biggr(\frac{1}{n},\pi \biggr] \biggr\}
\] 
 and 
\[\psi_n := P^U(A_n)\psi.
\]
Note that 
\begin{eqnarray*}
    \psi_n - \psi_n^*&=& P^U(A_n)\psi-P^U(\{1\})P^U(A_n)\psi\\ &=&  P^U(A_n)\psi-P^U(A_n)P^U(\{1\})\psi= P^U(A_n)(\psi-\psi^*),
\end{eqnarray*}
from which follows that $\supp(\mu_{\psi_n - \psi_n^*}^U) \subset A_n$. Moreover,
\[||\psi-\psi_n||^2 =   \int_{-\pi}^{\pi} |1- \chi_{A_n}(\theta)|^2 d(f^*\mu_{\psi - \psi^*}^U)(\theta) \to 0\]
as $n \to \infty$, by dominated convergence; so, $\psi_n \to \psi$ as $n \to \infty$. 

Now,  for each $n \in \mathbb{N}$,  
\begin{eqnarray*}\nonumber
  \bigg\|{\frac{1}{K}} \displaystyle\sum_{j=0}^{K-1} U^j\psi_n - \psi_n^* \bigg\|^2 &=&\frac{1}{K^2}   \int_{ \partial \mathbb{D}} \left| \frac{z^K -1}{z-1} \right|^2 d\mu_{\psi_n - \psi_n^*}^U(z)\\ \nonumber &\leq&   \frac{1}{ K^2}  \left(\sup_{z \in  \{e^{i \theta} \, : \, \theta \in (-\pi,-\frac{1}{n}] \cup (\frac{1}{n},\pi ] \}} \left| \frac{z^K -1}{z-1} \right|^2 \right) \mu_{\psi_n - \psi_n^*}^U(\partial \mathbb{D}) \\ \nonumber &=&   \frac{1}{ K^2}  \left( \sup_{ \theta \in (-\pi,-\frac{1}{n}] \cup  (\frac{1}{n},\pi ]} \left| \frac{\sin(K\theta/2)}{\sin(\theta/2)} \right|^2 \right)  \Vert\psi_n - \psi_n^*\Vert^2 \\ \nonumber &\le& \frac{1}{ (K\sin(1/2n))^2}\,  \Vert(I - P^U({\{1\}}))\psi_n\Vert^2\\  &\leq&   \frac{16n^2}{K^2}  \, \Vert\psi_n\Vert^2,
\end{eqnarray*}
where we have used that $1/(4n)\le \sin(1/(2n))$. Hence, for each $n \in \mathbb{N}$,  
\begin{eqnarray*}
\liminf_{K \to \infty} K^{2-\epsilon}\, \bigg\|{\frac{1}{K}} \displaystyle\sum_{j=0}^{K-1} U^j \psi_n - \psi_n^* \bigg\|^2  = 0,
\end{eqnarray*}
from which follows that $G^-(\epsilon)$ is dense in $\mathcal{H}$.

\subsection{Proof of Corollary \ref{corollaymain}} 

Let $\psi \in \mathcal{G}(U)$. If $d^-_{\mu_{\psi- \psi^*}^U}(1)>0$, then it follows from the definition of $d^-_{\mu_{\psi- \psi^*}^U}(1)$ that for each $0< \delta< d^-_{\mu_{\psi- \psi^*}^U}(1)$, there exists $\epsilon_\delta$ such that, for each $0<\epsilon <\epsilon_\delta$,
\begin{equation*}
\mu_{\psi- \psi^*}^U(A_\epsilon)  \leq C_\psi \epsilon^{d^-_{\mu_{\psi- \psi^*}^U}(1)-\frac{\delta}{2}} \leq C_\psi \epsilon^{\min\{d^-_{\mu_{\psi- \psi^*}^U}(1),2\} -\frac{\delta}{2}}.
\end{equation*}
But then, it follows from Theorem \ref{russiantheorem} i) that  
\begin{eqnarray*}
\limsup_{K \to \infty} K^{\min\big\{d^-_{\mu_{\psi- \psi^*}^U}(1),2\big\} -\delta}\, \bigg\|{\frac{1}{K}} \displaystyle\sum_{j=0}^{K-1} U^j \psi - \psi^* \bigg\|^2 = 0,
\end{eqnarray*}
which contradicts \eqref{eqeqeq0101a}. Therefore, $d^-_{\mu_{\psi- \psi^*}^U}(1)=0$.

Now, if $d^+_{\mu_{\psi- \psi^*}^U}(1)<2$, let $0< \delta< 2- d^+_{\mu_{\psi- \psi^*}^U}(1)$. Then, it follows from Theorem \ref{maintheorem} \eqref{eq0010} that for every sufficiently large $K$,  
\begin{eqnarray*}
K^{-d^+_{\mu_{\psi- \psi^*}^U}(1)-\frac{\delta}{3}} \leq \bigg\| {\frac{1}{K}} \displaystyle\sum_{j=0}^{K-1} U^j\psi -  \psi^* \bigg\|^2, 
\end{eqnarray*}
and so
\begin{eqnarray*}
\liminf_{K \to \infty} K^{2-\frac{\delta}{3}}\, \bigg\|{\frac{1}{K}} \displaystyle\sum_{j=0}^{K-1} U^j \psi - \psi^* \bigg\|^2 &\geq& \liminf_{K \to \infty}  K^{2-d^+_{\mu_{\psi- \psi^*}^U}(1)-\frac{2\delta}{3}}\\ &\geq& \liminf_{K \to \infty}  K^{\frac{\delta}{3}} = \infty.
\end{eqnarray*}
Since this result contradicts \eqref{eqeqeq0101b}, one concludes that $d^+_{\mu_{\psi- \psi^*}^U}(1)\geq 2$.


\begin{appendices}

\section{Appendix}\label{appendix}

Here, we present the proofs of Lemma \ref{mainlemma} and Theorem \ref{spectralgap}.

\subsection{ Proof of Lemma~\ref{mainlemma}}

For each $K \in\mathbb{N}$ and each $z \in \partial\mathbb{D}\setminus\{1\}$, recall that
\begin{eqnarray*}
 \displaystyle\sum_{j=0}^{K-1}  z^j &=& \frac{z^K -1}{z-1}.
\end{eqnarray*}
Note that, for each $\psi \in \mathcal{H}$,  $\mu_{\psi - \psi^*}^U (\{1\})=0$. Thus, by the Spectral Theorem, it follows that for each $\psi \in \mathcal{H}$ and each $K \in\mathbb{N}$,
\begin{eqnarray*}
  \bigg\|{\frac{1}{K}} \displaystyle\sum_{j=0}^{K-1} U^j\psi -  \psi^* \bigg\|^2 &=& 
  \bigg\|{\frac{1}{K}} \displaystyle\sum_{j=0}^{K-1}  U^j ( \psi - \psi^*) \bigg\|^2\\ &=& \bigg\|{\frac{1}{K}} \displaystyle\sum_{j=0}^{K-1} \int_{ \partial \mathbb{D}} z^j dP^U(z)( \psi - \psi^*)\bigg\|^2\\ &=&  \bigg\|{\frac{1}{K}} \int_{ \partial \mathbb{D} } \frac{z^K -1}{z-1} dP^U(z)(\psi - \psi^*)\bigg\|^2\\ &=&\frac{1}{K^2}   \int_{ \partial \mathbb{D} } \bigg| \frac{z^K -1}{z-1} \bigg|^2 d\mu_{\psi - \psi^*}^U(z).
\end{eqnarray*}

\subsection{ Proof of Theorem~\ref{spectralgap}}

 Since, for each $\psi \in \mathcal{H}$, $\supp(\mu_{\psi - \psi^*}^U) \subset \sigma(U) \subset \{1\} \cup \{e^{i\theta} \mid \theta \in (-\pi,-\gamma] \cup (\gamma,\pi] \}$ and $\mu_{\psi - \psi^*}^U (\{1\})=0$, it follows from Lemma~\ref{mainlemma} that for each $\psi \in \mathcal{H}$ and each $K \in\mathbb{N}$,
\begin{eqnarray*}\label{eqspectralgap}\nonumber
  \bigg\|{\frac{1}{K}} \displaystyle\sum_{j=0}^{K-1} U^j\psi - \psi^* \bigg\|^2 &=&\frac{1}{K^2}   \int_{ \partial \mathbb{D}} \left| \frac{z^K -1}{z-1} \right|^2 d\mu_{\psi - \psi^*}^U(z)\\ \nonumber &\leq&   \frac{1}{ K^2}  \left(\sup_{z \in \{e^{i\theta} \mid \theta \in (-\pi,-\gamma] \cup (\gamma,\pi] \}} \left| \frac{z^K -1}{z-1} \right|^2 \right) \mu_{\psi - \psi^*}^U(\partial \mathbb{D}) \\ \nonumber &=&   \frac{1}{ K^2}  \left( \sup_{\theta \in (-\pi,-\gamma] \cup (\gamma,\pi] } \left| \frac{\sin(K\theta/2)}{\sin(\theta/2)} \right|^2 \right)  \Vert\psi - \psi^*\Vert^2 \\ \nonumber &\le& \frac{1}{ (K\sin(\gamma/2))^2}\,  \Vert(I - P^U({\{1\}}))\psi\Vert^2\\  &\leq&   \frac{16}{\gamma^2 K^2}  \, \Vert\psi\Vert^2,
\end{eqnarray*}
where we have used that $\gamma/4\le \sin(\gamma/2)$ for each $0<\gamma<\pi$. Hence, for every $K \in\mathbb{N}$,
\begin{eqnarray*}
\Bigg\|{\frac{1}{K}} \displaystyle\sum_{j=0}^{K-1} U^j - P^U(\{1\}) \Bigg\|  \leq \frac{4}{\gamma K}.
\end{eqnarray*} 

\end{appendices}


\backmatter

\bmhead{Acknowledgments} We thank the anonymous referee for valuable suggestions that have substantially improved the exposition of the manuscript.

\bmhead{Funding} S.\ L.\ Carvalho thanks the partial support by FAPEMIG (Minas Gerais state agency; Universal Project, under contract 001/17/CEX-APQ-00352-17) and C.\ R.\ de Oliveira thanks the partial support by CNPq (a Brazilian government agency, under contract 303689/2021-8).

\end{document}